\documentclass[a4paper, 10pt]{amsart}
\usepackage{amsmath}
\usepackage{amssymb}
\usepackage{enumerate}

\def\Aut{\operatorname{Aut}}

\def\Prop{\operatorname{Prop}}
\def\rank{\operatorname{rank}}

\newtheorem{cor}{Corollary}

\newtheorem{lem}[cor]{Lemma}
\newtheorem{thm}[cor]{Theorem}

\theoremstyle{remark}
\newtheorem{rem}[cor]{Remark}

\theoremstyle{definition}

\title[Proper mappings between Hartogs triangles]{Proper holomorphic mappings between generalized Hartogs triangles}
\author{Pawe\l{} Zapa\l owski}

\address{Faculty of Mathematics and Computer Science, Jagiellonian University, \L o\-ja\-sie\-wi\-cza 6, 30-348 Krak\'ow, Poland}
\email{Pawel.Zapalowski@im.uj.edu.pl}
\thanks{The Author is partially supported by the Polish National Science Center (NCN) grant UMO-2014/15/D/ST1/01972}


\begin{document}

\begin{abstract}Answering all questions---concerning proper holomorphic mappings between generalized Hartogs triangles---posed by Jarnicki and Plfug (First steps in several complex variables: Reinhardt domains, 2008) we characterize the existence of proper holomorphic mappings between generalized Hartogs triangles and give their explicit form. In particular, we completely describe the group of holomorphic automorphisms of such domains and establish rigidity of proper holomorphic self-mappings on them.
\end{abstract}

\maketitle

\section{Introduction}
In the paper we study the proper holomorphic mappings between the generalized Hartogs triangles of equal dimensions (see definition below) giving full characterization of the existence of such mappings, their explicit form, and the complete description of the group of holomorphic automorphisms of such domains. Our results answer all questions posed by Jarnicki and Pflug in \cite{jarnicki2008}, Sections 2.5.2 and 2.5.3, concerning proper holomorphic mappings between generalized Hartogs triangles and holomorphic automorphisms of such domains.

Let us recall the definition of the above mentioned domains. Let $n,m\in\mathbb N$. For $p=(p_1,\dots,p_n)\in\mathbb R_{>0}^n$ and $q=(q_1,\dots,q_m)\in\mathbb R_{>0}^m$ define the \emph{generalized Hartogs triangle} as
\begin{equation*}
\mathbb F_{p,q}:=\Big\{(z,w)\in\mathbb C^n\times\mathbb C^m:\sum_{j=1}^n|z_j|^{2p_j}<\sum_{j=1}^m|w_j|^{2q_j}<1\Big\}.
\end{equation*}
Note that $\mathbb F_{p,q}$ is not smooth, pseudoconvex, Reinhardt domain, with the origin on the boundary. Moreover, if $n=m=1$, then $\mathbb F_{1,1}$ is the standard Hartogs triangle.

Let $p\in\mathbb R^n_{>0}$, $q\in\mathbb R^m_{>0}$ and $\tilde p\in\mathbb R^{\tilde n}_{>0}$, $\tilde q\in\mathbb R^{\tilde m}_{>0}$. We say that two generalized Hartogs triangles $\mathbb F_{p,q}$ and $\mathbb F_{\tilde p,\tilde q}$ are \emph{equidimensional}, if $n=\tilde n$ and $m=\tilde m$.

The problem of characterization of proper holomorphic mappings
\begin{equation}\label{eq:proper}
\mathbb F_{p,q}\longrightarrow\mathbb F_{\tilde p,\tilde q}
\end{equation}
and the group $\Aut(\mathbb F_{p,q})$ of holomorphic automorphisms of $\mathbb F_{p,q}$ has been investigated in many papers (see, e.g.,~\cite{landucci1989193}, \cite{chen2001177}, \cite{chen2002357}, \cite{chen200474}, \cite{chen2003215} for equidimensional case and \cite{chen2008557} for nonequidimensional case). It was Landucci, who considered the mappings (\ref{eq:proper}) first in 1989 as an example of proper holomorphic mappings between not smooth, pseudoconvex, Reinhardt domains, with the origin on the boundary, which do not satisfy a regularity property for the Bergman projection (the so-called $R$-condition). In \cite{landucci1989193} he gave complete characterization of the existence and the explicit form of the mappings (\ref{eq:proper}) in case $m=1$, $p,\tilde p\in\mathbb N^n$, and $q,\tilde q\in\mathbb N$. Then, in 2001 Chen and Xu (cf.~\cite{chen2001177}) characterized the existence of the mappings (\ref{eq:proper}) in case $n>1$, $m>1$, $p,\tilde p\in\mathbb N^n$, and $q,\tilde q\in\mathbb N^m$. Next step was made one year later, when the same Authors fully described proper holomorphic self-mappings of $\mathbb F_{p,q}$ for $n>1$, $m>1$, $p\in\mathbb N^n$, and $q\in\mathbb N^m$ (cf.~\cite{chen2002357}). In the same year, Chen in \cite{chen200474} characterized the existence of the mappings (\ref{eq:proper}) in case $n>1$, $m>1$, $p,\tilde p\in\mathbb R_{>0}^n$, and $q,\tilde q\in\mathbb R_{>0}^m$. Finally, Chen and Liu in 2003 gave the explicit form of proper holomorphic mappings $\mathbb F_{p,q}\longrightarrow\mathbb F_{\tilde p,\tilde q}$ but only for $n>1$, $m>1$, $p,\tilde p\in\mathbb N^n$, and $q,\tilde q\in\mathbb N^m$ (cf.~\cite{chen2003215}).

We emphasize that Landucci considered only the case $m=1$ with exponents being positive integers, whereas Chen, Xu, and Liu obtained some partial results with positive integer or arbitrary real positive exponents under general assumption $n\geq 2$ and $m\geq2$. Consequently, their results are far from being conclusive for the general case.

The main aim of this note is to give complete characterization of the existence of mappings (\ref{eq:proper}), where $n,m\in\mathbb N$, $p,\tilde p\in\mathbb R^n_{>0}$, $q,\tilde q\in\mathbb R^m_{>0}$, their explicit form, and the description of the group $\Aut(\mathbb F_{p,q})$ (cf.~Theorems~\ref{thm:n1m1}, \ref{thm:n1m2}, \ref{thm:n2m1}, and \ref{thm:n2m2}) for arbitrary dimensions and arbitrary positive real exponents. In particular, we obtain a classification theorem on rigidity of proper holomorphic self-mappings of generalized Hartogs triangles (cf.~Corollary~\ref{cor:rigid}).

It is worth pointing out that in the general case neither Landucci's method from \cite{landucci1989193} (where the assumption $p,\tilde p\in\mathbb N^n$, $q,\tilde q\in\mathbb N$ is essential) nor Chen's approach from \cite{chen200474} (where the proof strongly depends on the assumption $m\geq2$) can be used.

The paper is organized as follows. We start with stating the main results. For the convenience of the Reader we split them into four theorems with respect to the dimensions of the relevant parts of $\mathbb F_{p,q}$. Next we shall discuss the proper holomorphic mappings between complex ellipsoids (cf.~Section~\ref{sect:ce}) which will turn out to be quite useful in the sequel and may be interesting in its own right. The boundary behavior of the the mappings (\ref{eq:proper}) will also be studied. In the last section, making use of the description of proper holomorphic mappings between complex ellipsoids (Theorem~\ref{thm:ep}) and the boundary behavior of proper holomorphic mappings between generalized Hartogs triangles (Lemma~\ref{lem:kl}), we shall prove our main results.

Here is some notation. Throughout the paper  $\mathbb D$ denotes the unit disc in the complex plane, additionally by $\mathbb T$ we shall denote the unit circle, $\partial D$ stands for the boundary of the bounded domain $D\subset\mathbb C^n$. Let $\Sigma_n$ denote the group of the permutations of the set $\{1,\dots,n\}$. For $\sigma\in\Sigma_n,\ z=(z_1,\dots,z_n)\in\mathbb C^n$ denote $z_{\sigma}:=(z_{\sigma(1)},\dots,z_{\sigma(n)})$ and $\Sigma_n(z):=\{\sigma\in\Sigma_n:z_{\sigma}=z\}$. We shall also write $\sigma(z):=z_{\sigma}$. For $\alpha=(\alpha_1,\dots,\alpha_n)\in\mathbb R^n_{>0}$ and $\beta=(\beta_1,\dots,\beta_n)\in\mathbb R^n_{>0}$ we shall write $\alpha\beta:=(\alpha_1\beta_1,\dots,\alpha_n\beta_n)$ and $1/\beta:=(1/\beta_1,\dots,1/\beta_n)$. If, moreover, $\alpha\in\mathbb N^n$, then
$$
\Psi_{\alpha}(z):=z^{\alpha}:=(z_1^{\alpha_1},\dots,z_n^{\alpha_n}),\quad z=(z_1,\dots,z_n)\in\mathbb C^n.
$$
For $\lambda\in\mathbb C$, $A\subset\mathbb C^n$ let $\lambda A:=\{\lambda a:a\in A\}$ and $A_*:=A\setminus\{0\}$. Finally, let $\mathbb U(n)$ denote the set of unitary mappings $\mathbb C^n\longrightarrow\mathbb C^n$.

\section{Main results}

We start with the generalized Hartogs triangles of lowest dimension.

\begin{thm}\label{thm:n1m1}Let $n=m=1$, $p,q,\tilde p,\tilde q\in\mathbb R_{>0}$.
\begin{enumerate}[(a)]
\item\label{item:n1m1exist}There exists a proper holomorphic mapping $\mathbb F_{p,q}\longrightarrow\mathbb F_{\tilde p,\tilde q}$ if and only if there exist $k,l\in\mathbb N$ such that $$
    \frac{l\tilde q}{\tilde p}-\frac{kq}{p}\in\mathbb Z.
    $$
\item\label{item:n1m1form}A mapping $F:\mathbb F_{p,q}\longrightarrow\mathbb F_{\tilde p,\tilde q}$ is proper and holomorphic if and only if
    $$
    F(z,w)=\begin{cases}\left(\zeta z^kw^{l\tilde q/\tilde p-kq/p},\xi w^l\right),\quad&\textnormal{if }q/p\notin\mathbb Q\\\left(\zeta z^{k'}w^{l\tilde q/\tilde p-k'q/p}B\left(z^{p'}w^{-q'}\right),\xi w^l\right),\hfill&\textnormal{if }q/p\in\mathbb Q\end{cases},\quad(z,w)\in\mathbb F_{p,q},
    $$
    where $\zeta,\xi\in\mathbb T$, $k,l\in\mathbb N$, $k'\in\mathbb N\cup\{0\}$ are such that $l\tilde q/\tilde p-kq/p\in\mathbb Z$, $l\tilde q/\tilde p-k'q/p\in\mathbb Z$, $p',q'\in\mathbb N$ are relatively prime with $p/q=p'/q'$, and $B$ is a finite Blaschke product non-vanishing at 0 (if $B\equiv1$, then $k'>0$).

    In particular, there are non-trivial proper holomorphic self-mappings in $\mathbb F_{p,q}$.
\item\label{item:n1m1aut}$F\in\Aut(\mathbb F_{p,q})$ if and only if
    $$
    F(z,w)=\left(w^{q/p}\phi\left(zw^{-q/p}\right),\xi w\right),\quad(z,w)\in\mathbb F_{p,q},
    $$
    where $\xi\in\mathbb T$, and $\phi\in\Aut(\mathbb D)$ (moreover, $\phi(0)=0$ whenever $q/p\notin\mathbb N$).
\end{enumerate}
\end{thm}

\begin{rem}(a) The counterpart of the Theorem~\ref{thm:n1m1} for $p,q,\tilde p,\tilde q\in\mathbb N$ was proved (with minor mistakes) in \cite{landucci1989193}, where it was claimed that a mapping $F:\mathbb F_{p,q}\longrightarrow\mathbb F_{\tilde p,\tilde q}$ is proper and holomorphic if and only if
\begin{equation}\label{eq:landucci}
F(z,w)=\begin{cases}\left(\zeta z^kw^{l\tilde q/\tilde p-kq/p},\xi w^l\right),\quad&\textnormal{if }q/p\notin\mathbb N,\ l\tilde q/\tilde p-kq/p\in\mathbb Z\\\left(\zeta w^{l\tilde q/\tilde p}B\left(zw^{-q/p}\right),\xi w^l\right),\hfill&\textnormal{if }q/p\in\mathbb N,\ l\tilde q/\tilde p\in\mathbb N\end{cases},
\end{equation}
where $\zeta,\xi\in\mathbb T$, $k,l\in\mathbb N$, and $B$ is a finite Blaschke product. Nevertheless, the mapping
$$
\mathbb F_{2,3}\ni(z,w)\longmapsto\left(z^3w^3B\left(z^2w^{-3}\right),w^3\right)\in\mathbb F_{2,5},
$$
where $B$ is non-constant finite Blaschke product non-vanishing at 0, is proper holomorphic but not of the form (\ref{eq:landucci}). In fact, from the Theorem~\ref{thm:n1m1}~(\ref{item:n1m1form}) it follows immediately that for any choice of $p,q,\tilde p,\tilde q\in\mathbb N$ one may find a proper holomorphic mapping $F:\mathbb F_{p,q}\longrightarrow\mathbb F_{\tilde p,\tilde q}$ having, as a factor of the first component, non-constant Blaschke product non-vanishing at 0.

(b) Theorems~\ref{thm:n1m1}~(\ref{item:n1m1exist}), (\ref{item:n1m1form}) give a positive answer (modulo Landucci's mistake) to the question posed by Jarnicki and Pflug (cf.~\cite{jarnicki2008}, Remark~2.5.22\,(a)).

(c) Theorem~\ref{thm:n1m1}~(\ref{item:n1m1aut}) gives a positive answer to the question posed by Jarnicki and Pflug (cf.~\cite{jarnicki2008}, Remark~2.5.15\,(b)) in case $n=1$.
\end{rem}

\begin{thm}\label{thm:n1m2}Let $n=1$, $m\geq2$, $p,\tilde p\in\mathbb R_{>0}$, $q,\tilde q\in\mathbb R^m_{>0}$.
\begin{enumerate}[(a)]
\item\label{item:n1m2exist}There exists a proper holomorphic mapping $\mathbb F_{p,q}\longrightarrow\mathbb F_{\tilde p,\tilde q}$ if and only if there exists $\sigma\in\Sigma_m$ such that
    $$
    \frac{p}{\tilde p}\in\mathbb N\quad\text{and}\quad\frac{q_{\sigma}}{\tilde q}\in\mathbb N^m.
    $$
\item\label{item:n1m2form}A mapping $F:\mathbb F_{p,q}\longrightarrow\mathbb F_{\tilde p,\tilde q}$ is proper and holomorphic if and only if
    $$
    F(z,w)=(\zeta z^k,h(w)),\quad(z,w)\in\mathbb F_{p,q},
    $$
    where $\zeta\in\mathbb T$, $k\in\mathbb N$, and $h:\mathbb E_q\longrightarrow\mathbb E_{\tilde q}$ is proper and holomorphic such that $h(0)=0$ (cf.~Theorem~\ref{thm:ep}).

In particular, there are non-trivial proper holomorphic self-mappings in $\mathbb F_{p,q}$.
\item\label{item:n1m2aut}$F\in\Aut(\mathbb F_{p,q})$ if and only if
    $$
    F(z,w)=(\zeta z,h(w)),\quad(z,w)\in\mathbb F_{p,q},
    $$
    where $\zeta\in\mathbb T$, $h\in\Aut(\mathbb E_q)$, $h(0)=0$ (cf.~Theorem~\ref{thm:ep}).
\end{enumerate}
\end{thm}

Our next result is the following

\begin{thm}\label{thm:n2m1}Let $n\geq2$, $m=1$, $p=(p_1,\dots,p_n),\tilde p=(\tilde p_1,\dots,\tilde p_n)\in\mathbb R_{>0}^n$, $q,\tilde q\in\mathbb R_{>0}$.
\begin{enumerate}[(a)]
\item\label{item:n2m1exist}There exists a proper holomorphic mapping $\mathbb F_{p,q}\longrightarrow\mathbb F_{\tilde p,\tilde q}$ if and only if there exist $\sigma\in\Sigma_n$ and $r\in\mathbb N$ such that
    $$
    \frac{p_{\sigma}}{\tilde p}\in\mathbb N^n\quad\text{and}\quad\frac{r\tilde q-q}{\tilde p_j}\in\mathbb Z,\quad j=1,\dots,n.
    $$
\item\label{item:n2m1form}A mapping $F=(G_1,\dots,G_n,H):\mathbb F_{p,q}\longrightarrow\mathbb F_{\tilde p,\tilde q}$ is proper and holomorphic if and only if
    \begin{equation*}
    \begin{cases}G_j(z,w)=w^{r\tilde q/\tilde p_j}f_j\left(z_1w^{-q/p_1},\dots,z_nw^{-q/p_n}\right),\quad j=1,\dots,n,\\ H(z,w)=\xi w^r,\end{cases},\quad (z,w)\in\mathbb F_{p,q},
    \end{equation*}
    where $(f_1,\dots,f_n):\mathbb E_p\longrightarrow\mathbb E_{\tilde p}$ is proper and holomorphic (cf.~Theorem~\ref{thm:ep}), $\xi\in\mathbb T$, and $r\in\mathbb N$ is such that $(r\tilde q-q)/\tilde p_j\in\mathbb Z$, $j=1,\dots,n$. Moreover, if there is a $j$ such that $1/\tilde p_j\in\mathbb N$, then $q\in\mathbb N$ and $r\tilde q/\tilde p_j\in\mathbb N$ whenever $1/\tilde p_j\in\mathbb N$.

In particular, there are non-trivial proper holomorphic self-mappings in $\mathbb F_{p,q}$.
\item\label{item:n2m1aut}$F=(G_1,\dots,G_n,H)\in\Aut(\mathbb F_{p,q})$ if and only if
    \begin{equation*}
    \begin{cases}G_j(z,w)=w^{q/p_j}g_j\left(z_1w^{-q/p_1},\dots,z_nw^{-q/p_n}\right),\quad j=1,\dots,n,\\ H(z,w)=\xi w,\end{cases},\quad (z,w)\in\mathbb F_{p,q},
    \end{equation*}
    where $(g_1,\dots,g_n)\in\Aut(\mathbb E_p)$ (cf.~Theorem~\ref{thm:ep}), $\xi\in\mathbb T$.
\end{enumerate}
\end{thm}

\begin{rem}(a) Theorem~\ref{thm:n2m1}~(\ref{item:n2m1exist}) gives a positive answer to the question posed by Jarnicki and Pflug (cf.~\cite{jarnicki2008}, Remark~2.5.22\,(a)) in case $n\geq2$.

(b) Theorem~\ref{thm:n2m1}~(\ref{item:n2m1aut}) gives a positive answer to the question posed by Jarnicki and Pflug (cf.~\cite{jarnicki2008}, Remark~2.5.15\,(b)) in case $n\geq2$.

(c) It should be mentioned, that although the structure of the automorphism group $\Aut(\mathbb F_{p,q})$ does not change when passing from $p\in\mathbb N^n$, $q\in\mathbb N$ to $p\in\mathbb R_{>0}^n$, $q>0$, the class of proper holomorphic mappings $\mathbb F_{p,q}\longrightarrow\mathbb F_{\tilde p,\tilde q}$ does. It is a consequence of the fact that the structure of the proper holomorphic mappings $\mathbb E_p\longrightarrow\mathbb E_{\tilde p}$ changes when passing from $p,\tilde p\in\mathbb N^n$ to $p,\tilde p\in\mathbb R_{>0}^n$ (see Section~\ref{sect:ce}).
\end{rem}

\begin{thm}\label{thm:n2m2}Let $n,m\geq2$, $p,\tilde p\in\mathbb R^n_{>0}$, $q,\tilde q\in\mathbb R^m_{>0}$.
\begin{enumerate}[(a)]
\item\label{item:n2m2exist}There exists a proper holomorphic mapping $\mathbb F_{p,q}\longrightarrow\mathbb F_{\tilde p,\tilde q}$ if and only if there exist $\sigma\in\Sigma_n$ and $\tau\in\Sigma_m$ such that
    $$
    \frac{p_{\sigma}}{\tilde p}\in\mathbb N^n\quad\text{and}\quad\frac{q_{\tau}}{\tilde q}\in\mathbb N^m.
    $$
\item\label{item:n2m2form}A mapping $F:\mathbb F_{p,q}\longrightarrow\mathbb F_{\tilde p,\tilde q}$ is proper and holomorphic if and only if
    $$
    F(z,w)=(g(z),h(w)),\quad(z,w)\in\mathbb F_{p,q},
    $$
    where mappings $g:\mathbb E_p\longrightarrow\mathbb E_{\tilde p}$ and $h:\mathbb E_q\longrightarrow\mathbb E_{\tilde q}$ are proper and holomorphic such that $g(0)=0$, $h(0)=0$ (cf.~Theorem~\ref{thm:ep}).

In particular, every proper holomorphic self-mapping in $\mathbb F_{p,q}$ is an automorphism.
\item\label{item:n2m2aut}$F\in\Aut(\mathbb F_{p,q})$ if and only if
    $$
    F(z,w)=(g(z),h(w)),\quad(z,w)\in\mathbb F_{p,q},
    $$
    where $g\in\Aut(\mathbb E_p)$, $h\in\Aut(\mathbb E_q)$ with $g(0)=0$, $h(0)=0$ (cf.~Theorem~\ref{thm:ep}).
\end{enumerate}
\end{thm}

\begin{rem}(a) Theorem~\ref{thm:n2m2}~(\ref{item:n2m2exist}) was proved by Chen and Xu in \cite{chen2001177} (for $n,m\geq2$, $p,\tilde p\in\mathbb N^n$, $q,\tilde q\in\mathbb N^m$) and by Chen in \cite{chen200474} (for $n,m\geq2$, $p,\tilde p\in\mathbb R_{>0}^n$, $q,\tilde q\in\mathbb R_{>0}^m$).

(b) Theorems~\ref{thm:n1m2}~(\ref{item:n2m2form}), (\ref{item:n2m2aut}) were proved by Chen and Xu in \cite{chen2002357} for $n,m\geq2$, $p=\tilde p\in\mathbb N^n$, $q=\tilde q\in\mathbb N^m$.

(c) Theorem~\ref{thm:n2m2}~(\ref{item:n2m2aut}) gives an affirmative answer to the question posed by Jarnicki and Pflug (cf.~\cite{jarnicki2008}, Remark~2.5.17).
\end{rem}

A direct consequence of Theorems~\ref{thm:n1m1}, \ref{thm:n1m2}, \ref{thm:n2m1}, and \ref{thm:n2m2} is the following classification of rigid proper holomorphic self-mappings in generalized Hartogs triangles.

\begin{cor}\label{cor:rigid}Let $n,m\in\mathbb N$, $p\in\mathbb R^n_{>0}$, $q\in\mathbb R^m_{>0}$. Then any proper holomorphic self-mapping in $\mathbb F_{p,q}$ is an automorphism if and only if $n\geq2$ and $m\geq2$.
\end{cor}

\begin{rem}Corollary~\ref{cor:rigid} generalizes main result of \cite{chen2002357}, where it is proved that for $n\geq2$, $m\geq2$, $p\in\mathbb N^n$, and $q\in\mathbb N^m$ any proper holomorphic self-mapping in $\mathbb F_{p,q}$ is an automorphism. For more information on rigidity of proper holomorphic mappings between special kind of domains in $\mathbb C^n$, such as Cartan domains, Hua domains, etc., we refer the Reader to \cite{tu20021035}, \cite{tu200213}, \cite{tu2004310}, \cite{tu2014703}, and \cite{tu20151}.
\end{rem}

\section{Complex ellipsoids}\label{sect:ce}

In this section we discuss proper holomorphic mappings between complex ellipsoids. We shall exploit their form in the proofs of main results.

For $p=(p_1,\dots,p_n)\in\mathbb R_{>0}^n$, define the \emph{complex ellipsoid}
\begin{equation*}
\mathbb E_p:=\Big\{(z_1,\dots,z_n)\in\mathbb C^n:\sum_{j=1}^n|z_j|^{2p_j}<1\Big\}.
\end{equation*}

Note that $\mathbb E_{(1,\dots,1)}$ is the unit ball in $\mathbb C^n$. Moreover, if $p/q\in\mathbb N^n$, then $\Psi_{p/q}:\mathbb E_p\longrightarrow\mathbb E_q$ is proper and holomorphic.

The problem of characterization of proper holomorphic mappings between two given complex ellipsoids has been investigated in \cite{landucci1984807} and \cite{dini1991219}. The questions for the existence of such mappings as well as for its form in the case $p,q\in\mathbb N^n$ was completely solved by Landucci in 1984 (cf.~\cite{landucci1984807}). The case $p,q\in\mathbb R_{>0}^n$ was considered seven years later by Dini and Selvaggi Primicerio in \cite{dini1991219}, where the Authors characterized the existence of proper holomorphic mappings $\mathbb E_p\longrightarrow\mathbb E_q$ and found $\Aut(\mathbb E_p)$. They did not give, however, the explicit form of a proper holomorphic mapping between complex ellipsoids. Nevertheless, from the proof of Theorem~1.1 in \cite{dini1991219} we easily derive its form which shall be of great importance during the investigation of proper holomorphic mappings between generalized Hartogs triangles.

\begin{thm}\label{thm:ep}Assume that $n\geq2$, $p,q\in\mathbb R^n_{>0}$.
\begin{enumerate}[(a)]
\item(cf. \cite{landucci1984807}, \cite{dini1991219}). There exists a proper holomorphic mapping $\mathbb E_p\longrightarrow\mathbb E_q$ if and only if there exists $\sigma\in\Sigma_n$ such that
    $$
    \frac{p_{\sigma}}{q}\in\mathbb N^n.
    $$
\item\label{item:epform}A mapping $F:\mathbb E_p\longrightarrow\mathbb E_q$ is proper and holomorphic if and only if
    $$
    F=\Psi_{p_{\sigma}/(qr)}\circ\phi\circ \Psi_r\circ\sigma,
    $$
    where $\sigma\in\Sigma_n$ is such that $p_{\sigma}/q\in\mathbb N^n$, $r\in\mathbb N^n$ is such that $p_{\sigma}/(qr)\in\mathbb N^n$, and $\phi\in\Aut(\mathbb E_{p_{\sigma}/r})$.

    In particular, every proper holomorphic self-mapping in $\mathbb E_p$ is an automorphism.
\item\label{item:epaut}(cf. \cite{landucci1984807}, \cite{dini1991219}). If $0\leq k\leq n$, $p\in\{1\}^k\times(\mathbb R_{>0}\setminus\{1\})^{n-k}$, $z=(z',z_{k+1},\dots,z_n)$, then $F=(F_1,\dots,F_n)\in\Aut(\mathbb E_p)$ if and only if
    $$
    F_j(z)=\begin{cases}H_j(z'),\quad&\textnormal{if }j\leq k\\\zeta_jz_{\sigma(j)}\left(\frac{\sqrt{1-\|a'\|^2}}{1-\langle z',a'\rangle}\right)^{1/p_{\sigma(j)}},\quad&\textnormal{if }j>k\end{cases},
    $$
    where $\zeta_j\in\mathbb T$, $j>k$, $H=(H_1,\dots,H_k)\in\Aut(\mathbb B_k)$, $a'=H^{-1}(0)$, and $\sigma\in\Sigma_n(p)$.
\end{enumerate}
\end{thm}

\begin{proof}[Proof of Theorem~\ref{thm:ep}]Parts (a) and (c) was proved in \cite{dini1991219}.

(b) Let $F=(F_1,\dots,F_n)\in\Prop(\mathbb E_p,\mathbb E_q)$. Following \cite{stein1972}, any automorphism $H=(H_1,\dots,H_n)\in\Aut(\mathbb B_n)$ is of the form
\begin{equation*}\label{eq:aut}
H_j(z)=\frac{\sqrt{1-\|a\|^2}}{1-\langle z,a\rangle}\sum_{k=1}^nh_{j,k}(z_k-a_k),\quad z=(z_1,\dots,z_n)\in\mathbb B_n,\ j=1,\dots,n,
\end{equation*}
where $a=(a_1,\dots,a_n)\in\mathbb B_n$ and $Q=[h_{j,k}]$ is an $n\times n$ matrix such that
\begin{equation*}
\bar Q(\mathbb I_n-\bar a{}^t\!a){}^t\!Q=\mathbb I_n,
\end{equation*}
where $\mathbb I_n$ is the unit $n\times n$ matrix, whereas $\bar A$ (resp.~${}^t\!A$) is the conjugate (resp.~transpose) of an arbitrary matrix $A$. In particular, $Q$ is unitary if $a=0$.

It follows from \cite{dini1991219} that there exists $\sigma\in\Sigma_n$ such that $p_{\sigma}/q\in\mathbb N^n$, $h_{j,\sigma(j)}\neq0$, and
\begin{equation}\label{eq:ds1.6}
F_j(z)=\left(\frac{\sqrt{1-\|a\|^2}}{1-\langle z^p,a\rangle}h_{j,\sigma(j)}z_{\sigma(j)}^{p_{\sigma(j)}}\right)^{1/q_j}
\end{equation}
whenever $1/q_j\notin\mathbb N$.

If $1/q_j\in\mathbb N$ then $F_j$ either is of the form (\ref{eq:ds1.6}), where $p_{\sigma(j)}/q_j\in\mathbb N$, or
\begin{equation*}
F_j(z)=\left(\frac{\sqrt{1-\|a\|^2}}{1-\langle z^p,a\rangle}\sum_{k=1}^nh_{j,k}(z_k^{p_k}-a_k)\right)^{1/q_j}
\end{equation*}
where $p_k\in\mathbb N$ for any $k$ such that $h_{j,k}\neq0$.

Consequently, if we define $r=(r_1,\dots,r_n)$ as
\begin{equation*}
r_j:=\begin{cases}p_{\sigma(j)},\quad&\textnormal{if }a_{\sigma(j)}\neq0\textnormal{ or there is }k\neq\sigma(j)\textnormal{ with }h_{j,k}\neq0\\
p_{\sigma(j)}/q_j,\quad&\textnormal{otherwise}
\end{cases},
\end{equation*}
then it is easy to see that $r\in\mathbb N^n$, $p_{\sigma}/(qr)\in\mathbb N^n$, and $F$ is as desired.
\end{proof}

\begin{rem}(a) The counterpart of Theorem~\ref{thm:ep}~(\ref{item:epform}) obtained by Landucci in \cite{landucci1984807} for $p,q\in\mathbb N^n$ states that a mapping $F:\mathbb E_p\longrightarrow\mathbb E_q$ is proper and holomorphic if and only if
\begin{equation}\label{eq:epn}
F=\phi\circ \Psi_{p_{\sigma}/q}\circ\sigma,
\end{equation}
where $\sigma\in\Sigma_n$ is such that $p_{\sigma}/q\in\mathbb N^n$ and $\phi\in\Aut(\mathbb E_q)$.

(b) In the general case (\ref{eq:epn}) is no longer true (take, for instance, $\Psi_{(2,2)}\circ H\circ \Psi_{(2,2)}:\mathbb E_{(2,2)}\longrightarrow\mathbb E_{(1/2,1/2)}$, where $H\in\Aut(\mathbb B_2)$, $H(0)\neq0$). In particular, Theorem~\ref{thm:ep}~(\ref{item:epform}) gives a negative answer to the question posed by Jarnicki and Pflug (cf.~\cite{jarnicki2008}, Remark~2.5.20).

(c) Note that in the case $p,q\in\mathbb N^n$ we have $1/q_j\in\mathbb N$ if and only if $q_j=1$. Hence the above definition of $r$ implies that $r=p_{\sigma}/q$ and, consequently, Theorem~\ref{thm:ep}~(\ref{item:epform}) reduces to the Landucci's form (\ref{eq:epn}).

(d) Theorem~\ref{thm:ep}~(\ref{item:epaut}) gives a positive answer to the question posed by Jarnicki and Pflug (cf.~\cite{jarnicki2008}, Remark~2.5.11).
\end{rem}

\section{Boundary behavior of proper holomorphic mappings between Hartogs triangles}

Note that the boundary $\partial\mathbb F_{p,q}$ of the generalized Hartogs triangle $\mathbb F_{p,q}$ may be written as $\partial\mathbb F_{p,q}=\{0,0\}\cup K_{p,q}\cup L_{p,q}$, where
\begin{align*}
K_{p,q}:=&\Big\{(z,w)\in\mathbb C^n\times\mathbb C^m:0<\sum_{j=1}^n|z_j|^{2p_j}=\sum_{j=1}^m|w_j|^{2q_j}<1\Big\},\\
L_{p,q}:=&\Big\{(z,w)\in\mathbb C^n\times\mathbb C^m:\sum_{j=1}^n|z_j|^{2p_j}<\sum_{j=1}^m|w_j|^{2q_j}=1\Big\}.
\end{align*}

Let $\mathbb F_{p,q}$ and $\mathbb F_{\tilde p,\tilde q}$ be two generalized Hartogs triangles and let $F:\mathbb F_{p,q}\longrightarrow\mathbb F_{\tilde p,\tilde q}$ be proper holomorphic mapping. It is known (\cite{landucci1989193}, \cite{chen2001177}) that $F$ extends holomorphically through any boundary point $(z_0,w_0)\in\partial\mathbb F_{p,q}\setminus\{(0,0)\}$.

The aim of this section is to prove the following crucial fact.

\begin{lem}\label{lem:kl}Let $nm\neq1$. If $F:\mathbb F_{p,q}\longrightarrow\mathbb F_{\tilde p,\tilde q}$ is proper and holomorphic, then
\begin{equation*}
F(K_{p,q})\subset K_{\tilde p,\tilde q},\quad F(L_{p,q})\subset L_{\tilde p,\tilde q}.
\end{equation*}
\end{lem}

\begin{rem}Particular cases of Lemma~\ref{lem:kl} have already been proved by Landucci (cf.~\cite{landucci1989193}, Proposition~3.2, for $p,\tilde p\in\mathbb N^n$, $q,\tilde q\in\mathbb N^m$, $m=1$) and Chen (cf.~\cite{chen200474}, Lemmas~2.1 and 2.3, for $p,\tilde p\in\mathbb R_{>0}^n$, $q,\tilde q\in\mathbb R_{>0}^m$, $m>1$). Therefore it suffices to prove Lemma~\ref{lem:kl} for $n\geq2$ and $m=1$. The main difficulty in carrying out this construction is that in this case both the method from \cite{landucci1989193} (where the assumption $p,\tilde p\in\mathbb N^n$, $q,\tilde q\in\mathbb N$ is essential) as well as the one from \cite{chen200474} (where the assumption $m\geq2$ is essential) breaks down. Invariance of two defined parts of boundary of the generalized Hartogs triangles with respect to the proper holomorphic mappings presents a more delicate problem and shall be solved with help of the notion of Levi flatness of the boundary.
\end{rem}

The following two lemmas will be needed in the proof of Lemma~\ref{lem:kl}.

\begin{lem}\label{lem:flat}If $n\geq2$ and $m=1$, then $K_{p,q}$ is not Levi flat at $(z,w)\in K_{p,q}$, where at lest two coordinates of $z$ are non-zero (i.e.~the Levi form of the defining function restricted to the complex tangent space is not degenerate at $(z,w)$).
\end{lem}

\begin{proof}[Proof of Lemma~\ref{lem:flat}]Let
$$
r(z,w):=\sum_{j=1}^n|z_j|^{2p_j}-|w|^{2q},\quad(z,w)\in\mathbb C^n\times\mathbb C.
$$
Note that $r$ is local defining function for the Hartogs domain $\mathbb F_{p,q}$ (in neighborhood of any boundary point from $K_{p,q}$). It is easily seen that its Levi form equals
\begin{multline*}
\mathcal Lr((z,w);(X,Y))=\sum_{j=1}^np_j^2|z_j|^{2(p_j-1)}|X_j|^2-q^2|w|^{2(q-1)}|Y|^2,\\ (z,w)\in K_{p,q},\ (X,Y)\in\mathbb C^n\times\mathbb C,
\end{multline*}
whereas the complex tangent space at $(z,w)\in K_{p,q}$ is given by
$$
T_{\mathbb C}(z,w)=\Big\{(X,Y)\in\mathbb C^n\times\mathbb C:Y=\frac{1}{q\overline{w}|w|^{2(q-1)}}\sum_{j=1}^np_j\overline{z}_j|z_j|^{2(p_j-1)}X_j\Big\}
$$
(recall that $w\neq0$).

Fix $(z,w)\in K_{p,q}$ such that at lest two coordinates of $z$ are non-zero. To see that the Levi form of $r$ restricted to the complex tangent space is not degenerate at $(z,w)$, it suffices to observe that for any $(X,Y)\in T_{\mathbb C}(z,w)$
$$
\mathcal Lr((z,w);(X,Y))=\frac{1}{|w|^{2q}}\sum_{1\leq j<k\leq n}|z_j|^{2(p_j-1)}|z_k|^{2(p_k-1)}\left|p_jz_kX_j-p_kz_jX_k\right|^2.
$$
\end{proof}

\begin{lem}\label{lem:rank}Let $D\subset\mathbb C^{n+1}$ and $V\subset\mathbb C^n$ be bounded domains, $a\in V$, and let $\Phi:V\longrightarrow\partial D$ be holomorphic mapping such that $\rank\Phi'(a)=n$. Assume that $D$ has local defining function $r$ of class $\mathcal C^2$ in the neighborhood of $\Phi(a)$. Then $\partial D$ is Levi flat at $\Phi(a)$.
\end{lem}

\begin{proof}[Proof of Lemma~\ref{lem:rank}]Equality $r(\Phi(z))=0$, $z=(z_1,\dots,z_n)\in V$, implies
\begin{equation}\label{eq:dif1}
\sum_{j=1}^{n+1}\frac{\partial r}{\partial z_j}(\Phi(z))\frac{\partial\Phi_j}{\partial z_m}(z)=0,\quad z\in V,\ m=1,\dots,n,
\end{equation}
i.e.~
$$
X_m(z):=\left(\frac{\partial\Phi_1}{\partial z_m}(z),\dots,\frac{\partial\Phi_{n+1}}{\partial z_m}(z)\right)\in T_{\mathbb C}(\Phi(z)),\quad z\in V,\ m=1,\dots,n.
$$
Differentiating equality (\ref{eq:dif1}) with respect to $\overline{z}_m$ we get
$$
\sum_{j,k=1}^{n+1}\frac{\partial^2r}{\partial z_j\partial\overline{z}_k}(\Phi(z))\frac{\partial\Phi_j}{\partial z_m}(z)\overline{\frac{\partial\Phi_j}{\partial z_m}(z)}=0,\quad z\in V,\ m=1,\dots,n.
$$
Last equality for $z=a$ gives
\begin{equation}\label{eq:levi}
\mathcal Lr(\Phi(a);X_m(a))=0,\quad m=1,\dots,n.
\end{equation}
On the other hand, $\rank\Phi'(a)=n$ implies that the vectors $X_m(a)$, $m=1,\dots,n$, form the basis of the complex tangent space $T_{\mathbb C}(\Phi(a))$. Consequently, (\ref{eq:levi}) implies that $\mathcal Lr(\Phi(a);X)=0$ for any  $X\in T_{\mathbb C}(\Phi(a))$, i.e.~$\partial D$ is Levi flat at $\Phi(a)$.
\end{proof}

\begin{proof}[Proof of Lemma~\ref{lem:kl}]In view of Lemmas 2.1 and 2.3 from \cite{chen200474} it suffices to consider the case $n\geq2$ and $m=1$.

First we show that $F(L_{p,q})\subset L_{\tilde p,\tilde q}$. Suppose the contrary. Then $F(L_{p,q})\cap K_{\tilde p,\tilde q}\neq\varnothing$ or $(0,0)\in F(L_{p,q})$. First assume $F(L_{p,q})\cap K_{\tilde p,\tilde q}\neq\varnothing$. Since $L_{p,q}\setminus Z(J_F)$ is a dense open set of $L_{p,q}$, the continuity of $F$ implies that there is a point $(z_0,w_0)\in L_{p,q}\setminus Z(J_F)$ such that $F(z_0,w_0)\in K_{\tilde p,\tilde q}$. Without loss of generality we may assume that at least two coordinates of $G(z_0,w_0)$ are non-zero, where $F(z_0,w_0)=(G(z_0,w_0),H(z_0,w_0))\in\mathbb C^n\times\mathbb C$. Consequently, there is an open neighborhood $U\subset\mathbb C^n\times\mathbb C$ of $(z_0,w_0)$ such that $F|_U:U\longrightarrow F(U)$ is biholomorphic and $F(U\cap L_{p,q})=F(U)\cap K_{\tilde p,\tilde q}$. Take a neighborhood $V\subset\mathbb C^n$  of $z_0$ such that $(z,w_0)\in U\cap L_{p,q}$ for $z\in V$. Then
$$
V\ni z\overset{\Phi}\longmapsto F(z,w_0)\in F(U)\cap K_{\tilde p,\tilde q}
$$
is holomorphic mapping with $\rank\Phi'(z_0)=n$. By Lemma~\ref{lem:rank}, $K_{p,q}$ is Levi flat at $F(z_0,w_0)$, which contradicts Lemma~\ref{lem:flat}. The assumption $(0,0)\in F(L_{p,q})$ also leads to a contradiction. Indeed, one may repeat the reasoning from the proof of Lemma~2.1 from \cite{chen200474}.

Now we shall prove that $F(K_{p,q})\subset K_{\tilde p,\tilde q}$. Suppose the contrary. Then $F(K_{p,q})\cap L_{\tilde p,\tilde q}\neq\varnothing$ or $(0,0)\in F(K_{p,q})$. First assume $F(K_{p,q})\cap L_{\tilde p,\tilde q}\neq\varnothing$. Since $K_{p,q}\setminus Z(J_F)$ is a dense open set of $K_{p,q}$, the continuity of $F$ implies that there is a point $(z_0,w_0)\in K_{p,q}\setminus Z(J_F)$ such that $F(z_0,w_0)\in L_{\tilde p,\tilde q}$. Without loss of generality we may assume that at least two coordinates of $z_0$ are non-zero. Consequently, there is an open neighborhood $U\subset\mathbb C^n\times\mathbb C$ of $(z_0,w_0)$ such that $F|_U:U\longrightarrow F(U)$ is biholomorphic and $F(U\cap K_{p,q})=F(U)\cap L_{\tilde p,\tilde q}$. It remains to apply the previous reasoning to the inverse mapping $(F|_U)^{-1}:F(U)\longrightarrow U$. The assumption $(0,0)\in F(K_{p,q})$ also leads to a contradiction. Again, one may repeat the reasoning from the proof of Lemma~2.1 from \cite{chen200474} and therefore we skip it.
\end{proof}

\section{Proofs of the Theorems~\ref{thm:n1m1}, \ref{thm:n1m2}, \ref{thm:n2m1}, and \ref{thm:n2m2}}

In the proof of Theorem~\ref{thm:n1m1} we shall use part of the main result from \cite{isaev200633}, where complete characterization of not elementary proper holomorphic mappings between bounded Reinhardt domains in $\mathbb C^2$ is given (cf.~\cite{kosinski2009711} for unbounded case).

\begin{proof}[Proof of Theorem~\ref{thm:n1m1}]Observe, that (\ref{item:n1m1exist}) and (\ref{item:n1m1aut}) follows immediately from (\ref{item:n1m1form}).

If $F=(G,H)$ is of the form given in (\ref{item:n1m1form}), then it is holomorphic and
$$
|G(z,w)|^{\tilde p}|H(z,w)|^{-\tilde q}=\begin{cases}\left(|z||w|^{-q/p}\right)^{k\tilde p},\quad&\textnormal{if }q/p\notin\mathbb Q\\\left(|z||w|^{-q/p}\right)^{k'\tilde p}\left|B(z^{p'}w^{-q'})\right|^{\tilde p},\quad&\textnormal{if }q/p\in\mathbb Q\end{cases},
$$
i.e.~$F$ is proper.

On the other hand, let $F:\mathbb F_{p,q}\longrightarrow\mathbb F_{\tilde p,\tilde q}$ be arbitrary mapping which is proper and holomorphic.

Assume first that $F$ is elementary algebraic mapping, i.e. it is of the form
$$
F(z,w)=\left(\alpha z^aw^b,\beta z^cw^d\right),
$$
where $a,b,c,d\in\mathbb Z$ are such that $ad-bc\neq0$ and $\alpha,\beta\in\mathbb C$ are some constants. Since $F$ is surjective, we infer that $c=0$, $d\in\mathbb N$, and $\xi:=\beta\in\mathbb T$. Moreover,
\begin{equation}\label{eq:elem1}
|\alpha|^{\tilde p}|z|^{a\tilde p}|w|^{b\tilde p-d\tilde q}<1,
\end{equation}
whence $a\in\mathbb N$, $b\tilde p-d\tilde q\in\mathbb N$, and $\zeta:=\alpha\in\mathbb T$. Let $k:=a$, $l:=d$. One may rewrite (\ref{eq:elem1}) as
$$
\left(|z|^p|w|^{-q}\right)^{k\tilde p/p}|w|^{b\tilde p-l\tilde q+kq\tilde p/p}<1.
$$
Since one may take sequence $(z_{\nu},1/2)_{\nu\in\mathbb N}\subset\mathbb F_{p,q}$ with $|z_{\nu}|^p2^q\to1$ as $\nu\to\infty$, we infer that $b\tilde p-l\tilde q+kq\tilde p/p=0$, i.e.
$$
b=\frac{l\tilde q}{\tilde p}-\frac{kq}{p}.
$$
Consequently, $F$ is as in the Theorem~\ref{thm:n1m1}~(\ref{item:n1m1form}).

Assume now that $F$ is not elementary. Then it follows from the Theorem 0.1 in \cite{isaev200633} that $F$ is of the form
$$
F(z,w)=\left(\alpha z^aw^b\tilde B\left(z^{p'}w^{-q'}\right),\beta w^l\right),
$$
where $a,b\in\mathbb Z$, $a\geq0$, $p',q',l\in\mathbb N$, $p',q'$ are relatively prime,
\begin{equation}\label{eq:isakru}
\frac{q'}{p'}=\frac{q}{p},\quad\frac{\tilde q}{\tilde p}=\frac{aq'+bp'}{lp'},
\end{equation}
$\alpha,\beta\in\mathbb C$ are some constants, and $\tilde B$ is a non-constant finite Blaschke product non-vanishing at the origin.

From the surjectivity of $F$ we immediately infer that $\zeta:=\alpha\in\mathbb T$ and $\xi:=\beta\in\mathbb T$. If we put $k':=a$, then (\ref{eq:isakru}) implies
$$
b=\frac{l\tilde q}{\tilde p}-\frac{k'q}{p},
$$
which ends the proof.
\end{proof}

\begin{proof}[Proof of Theorem~\ref{thm:n1m2}]We shall write $w=(w_1,\dots,w_m)\in\mathbb C^m$. Without loss of generality we may assume that there is $0\leq\mu\leq m$ with $\tilde q\in\{1\}^{\mu}\times(\mathbb R_{>0}\setminus\{1\})^{m-\mu}$. Let
$$
F=(G,H):\mathbb F_{p,q}\longrightarrow\mathbb F_{\tilde p,\tilde q}\subset\mathbb C\times\mathbb C^m
$$
be proper holomorphic mapping. It follows from Lemma~\ref{lem:kl} that $F(L_{p,q})\subset L_{\tilde p,\tilde q}$. Moreover, by Lemma 2.2 from \cite{chen200474} (note that the proof remains valid for $n=1$), $H$ is independent of the variable $z$. Hence $h:=H(0,\cdot):(\mathbb E_q)_*\longrightarrow(\mathbb E_{\tilde q})_*$ is proper and holomorphic. Consequently, by Hartogs theorem, it extends to proper holomorphic mapping $h:\mathbb E_q\longrightarrow\mathbb E_{\tilde q}$, i.e.~(cf.~Theorem~\ref{thm:ep}~(\ref{item:epform}))
$$
h=\Psi_{q_{\sigma}/(\tilde qr)}\circ\psi\circ \Psi_r\circ\sigma
$$
for some $\sigma\in\Sigma_m$ with $q_{\sigma}/\tilde q\in\mathbb N^m$, $r\in\mathbb N^m$ with $q_{\sigma}/(\tilde qr)\in\mathbb N^m$, and $\psi\in\Aut(\mathbb E_{q_{\sigma}/r})$ with $\psi(0)=0$. Indeed, if $a=(a_1,\dots,a_m)$ is a zero of $h$ we immediately get
$$
G(z,a)=0,\quad|z|^{2p}<\sum_{j=1}^m|a_j|^{2q_j},
$$
which is clearly a contradiction, unless $a=0$. Consequently, $h(0)=0$.

Without loss of generality we may assume that there is $\mu\leq l\leq m$ with $1/\tilde q_j\notin\mathbb N$ if and only if $j>l$. It follows from the proof of Theorem~\ref{thm:ep}~(\ref{item:epform}) that
$$
\frac{q_{\sigma(j)}}{r_j}=\begin{cases}1,\quad&\textnormal{if }j=1,\dots,l\\\tilde q_j,\quad&\textnormal{if }j=l+1,\dots,m\end{cases},
$$
whence
$$
\psi(w)=(U(w_1,\dots,w_l),\xi_{l+1}w_{l+\tau(1)},\dots,\xi_mz_{l+\tau(m-l)}),
$$
where $U=(U_1,\dots,U_l)\in\mathbb U(l)$ and $\tau\in\Sigma_{m-l}(\tilde q_{l+1},\dots,\tilde q_m)$. Finally,
\begin{multline*}
h(w)=\left(U_1^{1/\tilde q_1}\left(w_{\sigma(1)}^{q_{\sigma(1)}},\dots,w_{\sigma(l)}^{q_{\sigma(l)}}\right),\dots, U_l^{1/\tilde q_l}\left(w_{\sigma(1)}^{q_{\sigma(1)}},\dots,w_{\sigma(l)}^{q_{\sigma(l)}}\right),\right.\\ \left.\xi_{l+1}w_{\sigma(l+1)}^{q_{\sigma(l+1)}/\tilde q_{l+1}},\dots,\xi_mw_{\sigma(m)}^{q_{\sigma(m)}/\tilde q_m}\right).
\end{multline*}
In particular, if we write $h=(h_1,\dots,h_m)$,
\begin{equation}\label{eq:unitarity}
\sum_{j=1}^m|h_j(w)|^{2\tilde q_j}=\sum_{j=1}^m|w_j|^{2q_j},\quad w\in\mathbb E_q.
\end{equation}

For $w\in\mathbb C^m$, $0<\rho_w:=\sum_{j=1}^m|w_j|^{2q_j}<1$ let
$$
g(z):=G(z,w),\quad z\in\rho_{w}^{1/(2p)}\mathbb D.
$$
$g$ may depend, a priori, on $w$. Since $F(K_{p,q})\subset K_{\tilde p,\tilde q}$ (cf.~Lemma~\ref{lem:kl}), it follows from (\ref{eq:unitarity}) that $g:\rho_w^{1/(2p)}\mathbb D\longrightarrow\rho_w^{1/(2\tilde p)}\mathbb D$ is proper and holomorphic, i.e.
\begin{equation}\label{eq:f1}
g(z)=\rho_w^{1/(2\tilde p)}B\left(z\rho_w^{-1/(2p)}\right),\quad z\in\rho_w^{1/(2p)}\mathbb D,
\end{equation}
where $B$ is a finite Blaschke product. Let
\begin{align*}
\mathbb F^0_{p,q}&:=\mathbb F_{p,q}\cap\left(\mathbb C\times\{0\}^{\sigma(1)-1}\times\mathbb C\times\{0\}^{m-\sigma(1)}\right),\\
\mathbb F^0_{\tilde p,q_{\sigma}/r}&:=\mathbb F_{\tilde p,q_{\sigma}/r}\cap\left(\mathbb C^2\times\{0\}^{m-1}\right).
\end{align*}

Let $\Phi\in\Aut(\mathbb F_{\tilde p,q_{\sigma}/r})$ be defined by
$$
\Phi(z,w):=\left(z,U^{-1}(w_1,\dots,w_l),w_{l+1},\dots,w_m\right)
$$
and let
$$
\hat\xi_1:=\begin{cases}\xi_1,\quad&\textnormal{if }l=0\\1,\quad&\textnormal{if }l>0\end{cases},\qquad\hat q_1:=\begin{cases}\tilde q_1,\quad&\textnormal{if }l=0\\1,\quad&\textnormal{if }l>0\end{cases}.
$$
Then $\Phi\circ(G,\psi\circ \Psi_r\circ\sigma):\mathbb F^0_{p,q}\longrightarrow\mathbb F^0_{\tilde p,q_{\sigma}/r}$ is proper and holomorphic with
\begin{equation}\label{eq:phi1}
(\Phi\circ(G,\psi\circ\Psi_r\circ\sigma))(z,w)=\left(G(z,w),\hat\xi_1w_{\sigma(1)}^{q_{\sigma(1)}/\hat q_1},0,\dots,0\right),\quad (z,w)\in\mathbb F^0_{p,q}.
\end{equation}

It follows from Theorem~\ref{thm:n1m1} that
\begin{equation}\label{eq:phithmfp2}
(\Phi\circ(G,\psi\circ\Psi_r\circ\sigma))(z,w)=\left(\hat G(z,w),\eta w_{\sigma(1)}^s,0,\dots,0\right),\quad (z,w)\in\mathbb F^0_{p,q},
\end{equation}
where
$$
\hat G(z,w):=\begin{cases}\zeta z^kw_{\sigma(1)}^{s\hat q_1/\tilde p-kq_{\sigma(1)}/p},\hfill&\textnormal{if }q_{\sigma(1)}/p\notin\mathbb Q\\\zeta z^{k'}w_{\sigma(1)}^{s\hat q_1/\tilde p-k'q_{\sigma(1)}/p}\hat B\left(z^{p'}w_{\sigma(1)}^{-q'_{\sigma(1)}}\right),\hfill&\textnormal{if }q_{\sigma(1)}/p\in\mathbb Q\end{cases},
$$
$\zeta,\eta\in\mathbb T$, $k,s,p',q'_{\sigma(1)}\in\mathbb N$, $k'\in\mathbb N\cup\{0\}$ are such that $p',q'_{\sigma(1)}$ are relatively prime, $q_{\sigma(1)}/p=q'_{\sigma(1)}/p'$, $s\hat q_1/\tilde p-kq_{\sigma(1)}/p\in\mathbb Z$, $q_{\sigma(1)}/p=q'_{\sigma(1)}/p'$, $s\hat q_1/\tilde p-kq_{\sigma(1)}/p\in\mathbb Z$, and $\hat B$ is a finite Blaschke product non-vanishing at 0 (if $\hat B\equiv1$ then $k'>0$). Hence
\begin{multline}\label{eq:h1}
(\Phi\circ(G,\psi\circ\Psi_r\circ\sigma))(z,w)=\\
\left(\hat G(z,w)+\alpha(z,w),w_{\sigma(1)}^{q_{\sigma(1)}},\dots,w_{\sigma(l)}^{q_{\sigma(l)}},\xi_{l+1}w_{\sigma(l+1)}^{q_{\sigma(l+1)}/\tilde q_{l+1}},\dots,\xi_mw_{\sigma(m)}^{q_{\sigma(m)}/\tilde q_m}\right),
\end{multline}
for $(z,w)\in\mathbb F_{p,q}$, $w_{\sigma(1)}\neq0$, where $\alpha$ is holomorphic on $\mathbb F_{p,q}$ with $\alpha|_{\mathbb F^0_{p,q}}=0$. Comparing (\ref{eq:phi1}) and (\ref{eq:phithmfp2}) we conclude that
$$
\eta=\hat\xi_1,\quad s=q_{\sigma(1)}/\hat q_1.
$$
Since the mapping on the left side of (\ref{eq:h1}) is holomorphic on $\mathbb F_{p,q}$, the function
\begin{equation}\label{eq:hatG}
\hat G(z,w)=\begin{cases}\zeta z^kw_{\sigma(1)}^{q_{\sigma(1)}(1/\tilde p-k/p)},\hfill&\textnormal{if }q_{\sigma(1)}/p\notin\mathbb Q\\\zeta z^{k'} w_{\sigma(1)}^{q_{\sigma(1)}(1/\tilde p-k'/p)}\hat B\left(z^{p'}w_{\sigma(1)}^{-q'_{\sigma(1)}}\right),\hfill&\textnormal{if }q_{\sigma(1)}/p\in\mathbb Q\end{cases}
\end{equation}
with $q_{\sigma(1)}(1/\tilde p-k/p)\in\mathbb Z$ and $q_{\sigma(1)}(1/\tilde p-k'/p)\in\mathbb Z$ has to be holomorphic on $\mathbb F_{p,q}$, too. Since $m\geq2$, it may happen $w_{\sigma(1)}=0$. Consequently, $q_{\sigma(1)}(1/\tilde p-k/p)\in\mathbb N\cup\{0\}$ in the first case of (\ref{eq:hatG}), whereas $\hat B(t)=t^{k''}$ for some $k''\in\mathbb N$ with $q_{\sigma(1)}(1/\tilde p-k'/p)-k''q'_{\sigma(1)}\in\mathbb N\cup\{0\}$ in the second case. Thus
$$
\hat G(z,w)=\zeta z^kw_{\sigma(1)}^{q_{\sigma(1)}(1/\tilde p-k/p)},
$$
where $k\in\mathbb N$, $q_{\sigma(1)}(1/\tilde p-k/p)\in\mathbb N\cup\{0\}$ (in the second case of (\ref{eq:hatG}) it suffices to take $k:=k'+p'k''$).

Observe that $\hat G+\alpha=G$. Fix $w\in\{0\}^{\sigma(1)-1}\times\mathbb C\times\{0\}^{m-\sigma(1)}$ with $0<\rho_w<1$. Then $\rho_w=|w_{\sigma(1)}|^{2q_{\sigma(1)}}$ and $\hat G(\cdot,w)=g$ on $\rho_w^{1/(2p)}\mathbb D$, i.e.
$$
\zeta z^kw_{\sigma(1)}^{q_{\sigma(1)}(1/\tilde p-k/p)}=|w_{\sigma(1)}|^{q_{\sigma(1)}/\tilde p}B\left(z|w_{\sigma(1)}|^{-q_{\sigma(1)}/p}\right),\quad z\in|w_{\sigma(1)}|^{q_{\sigma(1)}/p}\mathbb D.
$$
Hence $B(t)=\zeta t^k$ and $q_{\sigma(1)}(1/\tilde p-k/p)=0$, i.e.~$k=p/\tilde p$. Hence part (\ref{item:n1m2exist}) is proved. To finish part (\ref{item:n1m2form}), note that $g(z)=\zeta z^{p/\tilde p}$. Consequently, $g$ does not depend on $w$ and
$$
G(z,w)=\zeta z^{p/\tilde p},\quad (z,w)\in\mathbb F_{p,q}.
$$

Part (\ref{item:n1m2aut}) follows directly from (\ref{item:n1m2form}).
\end{proof}

\begin{proof}[Proof of Theorem~\ref{thm:n2m1}]Firstly, if $p,q,\tilde p$, and $\tilde q$ satisfy the condition in (\ref{item:n2m1exist}), then the mapping
$$
\mathbb F_{p,q}\ni(z_1,\dots,z_n,w)\longmapsto\left(z_{\sigma(1)}^{p_{\sigma(1)}/\tilde p_1}w^{(r\tilde q-q)/\tilde p_1},\dots,z_{\sigma(n)}^{p_{\sigma(n)}/\tilde p_n}w^{(r\tilde q-q)/\tilde p_n},w^r\right)\in\mathbb F_{\tilde p,\tilde q}
$$
is proper and holomorphic.

Secondly, if the mapping $F$ is defined by the formulas given in (\ref{item:n2m1form}), then, using Theorem~\ref{thm:ep}~(\ref{item:epform}), it is easy to see that $F:\mathbb F_{p,q}\longrightarrow\mathbb F_{\tilde p,\tilde q}$ is proper and holomorphic.

Finally, (\ref{item:n2m1aut}) is a direct consequence of (\ref{item:n2m1form}) and Theorem~\ref{thm:ep}~(\ref{item:epaut}).

Thus it remains to prove that if $F:\mathbb F_{p,q}\longrightarrow\mathbb F_{\tilde p,\tilde q}$ is proper and holomorphic, then $p,q,\tilde p$, and $\tilde q$ satisfy the condition in (\ref{item:n2m1exist}) and $F$ is given by formulas in (\ref{item:n2m1form}).

Let
$$
F=(G,H)=(G_1,\dots,G_n,H):\mathbb F_{p,q}\longrightarrow\mathbb F_{\tilde p,\tilde q}
$$
be proper holomorphic mapping. Since $F(L_{p,q})\subset L_{\tilde p,\tilde q}$ (cf.~Lemma~\ref{lem:kl}), it follows from the proof of Lemma~2.2 in \cite{chen200474} that $H$ does not depend on the variable $z$. Hence $h:=H(0,\cdot)$ is proper and holomorphic self-mapping in $\mathbb D_*$. Consequently, by Hartogs theorem, it extends to proper holomorphic mapping $h:\mathbb D\longrightarrow\mathbb D$, i.e.~$h$ is a finite Blaschke product. On the other hand, if $h(a)=0$ we immediately get
$$
G(z,a)=0,\quad\sum_{j=1}^n|z_j|^{2p_j}<|a|^{2q},
$$
which is clearly a contradiction, unless $a=0$. Hence
\begin{equation}\label{eq:H}
H(z,w)=\xi w^r
\end{equation}
for some $\xi\in\mathbb T$ and $r\in\mathbb N$.

For $w$, $0<|w|<1$, let
$$
\mathbb E_{p,q}(w):=\Big\{(z_1,\dots,z_n)\in\mathbb C^n:\sum_{j=1}^n|z_j|^{2p_j}<|w|^{2q}\Big\}.
$$
Since $F(K_{p,q})\subset K_{\tilde p,\tilde q}$ (cf.~Lemma~\ref{lem:kl}), it follows from (\ref{eq:H}) that $G(\cdot,w):\mathbb E_{p,q}(w)\longrightarrow\mathbb E_{\tilde p,r\tilde q}(w)$ is proper and holomorphic. Hence, if we put
$$
\hat f_j(z_1,\dots,z_n):=\xi w^{-r\tilde q/\tilde p_j}G_j\left(z_1w^{q/p_1},\dots,z_nw^{-q/p_n},w\right),\quad j=1,\dots,n,
$$
we conclude that $\hat f=(\hat f_1,\dots,\hat f_n):\mathbb E_p\longrightarrow\mathbb E_{\tilde p}$ is proper and holomorphic. It follows from the proof of Theorem 2 in \cite{andreotti1964249} that $\hat f$ does not depend on $w$. Consequently, for $f=(f_1,\dots,f_n)$, where $f_j:=\xi^{-1}\hat f_j$, $j=1,\dots,n$, we obtain
$$
G_j(z_1,\dots,z_n,w)=w^{r\tilde q/\tilde p_j}f_j\left(z_1w^{-q/p_1},\dots,z_nw^{-q/p_n}\right),\quad j=1,\dots,n.
$$
To complete the proof it remains to apply the explicit form of an $f$ (cf.~Theorem~\ref{thm:ep}~(\ref{item:epform})).
\end{proof}

We are left with the proof of Theorem~\ref{thm:n2m2}. Although its proof proceeds parallel to the one of Theorem~\ref{thm:n1m2}, we decided---for the convenience of the Reader---to present it due to some technical details that make both proofs different.

\begin{proof}[Proof of Theorem~\ref{thm:n2m2}]We will write $z=(z_1,\dots,z_n)\in\mathbb C^n$ and $w=(w_1,\dots,w_m)\in\mathbb C^m$. Without loss of generality we may assume that there is $0\leq\nu\leq n$ with $\tilde p\in\{1\}^{\nu}\times(\mathbb R_{>0}\setminus\{1\})^{n-\nu}$ and $0\leq\mu\leq m$ with $\tilde q\in\{1\}^{\mu}\times(\mathbb R_{>0}\setminus\{1\})^{m-\mu}$. Let
$$
F=(G,H):\mathbb F_{p,q}\longrightarrow\mathbb F_{\tilde p,\tilde q}\subset\mathbb C^n\times\mathbb C^m
$$
be proper holomorphic mapping. It follows from Lemma~\ref{lem:kl} that $F(L_{p,q})\subset L_{\tilde p,\tilde q}$ and hence, using Lemma~2.2 from \cite{chen200474}, $H$ is independent of the variable $z$. Hence the mapping $h:=H(0,\cdot):(\mathbb E_q)_*\longrightarrow(\mathbb E_{\tilde q})_*$ is proper and holomorphic. Consequently, by Hartogs theorem, it extends to proper and holomorphic mapping $h:\mathbb E_q\longrightarrow\mathbb E_{\tilde q}$, i.e.~(cf.~Theorem~\ref{thm:ep}~(\ref{item:epform}))
$$
h=\Psi_{q_{\tau}/(\tilde qt)}\circ\psi\circ\Psi_t\circ\tau
$$
for some $\tau\in\Sigma_m$ with $q_{\tau}/\tilde q\in\mathbb N^m$, $t\in\mathbb N^m$ with $q_{\tau}/(\tilde qt)\in\mathbb N^m$, and $\psi\in\Aut(\mathbb E_{q_{\tau}/t})$ with $\psi(0)=0$. Indeed, if $a=(a_1,\dots,a_m)$ is a zero of $h$, we immediately get
$$
G(z,a)=0,\quad\sum_{j=1}^n|z_j|^{2p_j}<\sum_{j=1}^m|a_j|^{2q_j},
$$
which is clearly a contradiction, unless $a=0$. Consequently, $h(0)=0$.

Without loss of generality we may assume that there is $\mu\leq l\leq m$ with $1/\tilde q_j\notin\mathbb N$ if and only if $j=l+1,\dots,m$. It follows from the proof of Theorem~\ref{thm:ep}~(\ref{item:epform}) that
$$
\frac{q_{\tau(j)}}{t_j}=\begin{cases}1,\quad&\textnormal{if }j=1,\dots,l\\\tilde q_j,\quad&\textnormal{if }j=l+1,\dots,m\end{cases},
$$
whence
$$
\psi(w)=(U(w_1,\dots,w_l),\xi_{l+1}w_{l+\omega(1)},\dots,\xi_mw_{l+\omega(m-l)}),
$$
where $U=(U_1,\dots,U_l)\in\mathbb U(l)$, $\xi_j\in\mathbb T$, $j>l$, and $\omega\in\Sigma_{m-l}(\tilde q_{l+1},\dots,\tilde q_m)$. Finally,
\begin{multline*}
h(w)=\left(U_1^{1/\tilde q_1}\left(w_{\tau(1)}^{q_{\tau(1)}},\dots,w_{\tau(l)}^{q_{\tau(l)}}\right),\dots, U_l^{1/\tilde q_l}\left(w_{\tau(1)}^{q_{\tau(1)}},\dots,w_{\tau(l)}^{q_{\tau(l)}}\right),\right.\\
\left.\xi_{l+1}w_{\tau(l+1)}^{q_{\tau(l+1)}/\tilde q_{l+1}},\dots,\xi_mw_{\tau(m)}^{q_{\tau(m)}/\tilde q_m}\right).
\end{multline*}
In particular, if we write $h=(h_1,\dots,h_m)$, then
\begin{equation}\label{eq:unitarity}
\sum_{j=1}^m|h_j(w)|^{2\tilde q_j}=\sum_{j=1}^m|w_j|^{2q_j},\quad w=(w_1,\dots,w_m)\in\mathbb E_q.
\end{equation}

For $w\in\mathbb C^m$, $0<\rho_w:=\sum_{j=1}^m|w_j|^{2q_j}<1$ let
$$
\mathbb E_{p,q}(w):=\Big\{z\in\mathbb C^n:\sum_{j=1}^n|z_j|^{2p_j}<\sum_{j=1}^m|w_j|^{2q_j}\Big\}.
$$
Since $F(K_{p,q})\subset K_{\tilde p,\tilde q}$ (cf.~Lemma~\ref{lem:kl}), it follows from (\ref{eq:unitarity}) that $g:=G(\cdot,w):\mathbb E_{p,q}(w)\longrightarrow\mathbb E_{\tilde p,q}(w)$ is proper and holomorphic. Note that $g$ may depend, a priori, on $w$.

Let
\begin{equation}\label{eq:fg}
f_j(z):=\rho_w^{-1/(2\tilde p_j)}g_j\left(z_1\rho_w^{1/(2p_1)},\dots,z_n\rho_w^{1/(2p_n)}\right),\quad j=1,\dots,n.
\end{equation}
Then $f:=(f_1,\dots,f_n):\mathbb E_p\longrightarrow\mathbb E_{\tilde p}$ is proper nad holomorphic, i.e.
\begin{equation}\label{eq:f}
f=\Psi_{p_{\sigma}/(\tilde ps)}\circ\varphi\circ\Psi_s\circ\sigma
\end{equation}
for some $\sigma\in\Sigma_n$ with $p_{\sigma}/\tilde p\in\mathbb N^n$, $s\in\mathbb N^n$ with $p_{\sigma}/(\tilde ps)\in\mathbb N^n$, and $\varphi\in\Aut(\mathbb E_{p_{\sigma}/s})$. Without loss of generality we may assume that there is $\nu\leq k\leq n$ such that $1/\tilde p_j\notin\mathbb N$ if and only if $j=k+1,\dots,n$. It follows from the proof of Theorem~\ref{thm:ep}~(\ref{item:epform}) that
$$
\frac{p_{\sigma(j)}}{s_j}=\begin{cases}1,\quad&\textnormal{if }j=1,\dots,k\\\tilde p_j,\quad&\textnormal{if }j=k+1,\dots,n\end{cases},
$$
whence
$$
\varphi(z)=\left(T(z_1,\dots,z_k),\zeta_{k+1}z_{k+\omega(1)},\dots,\zeta_nz_{k+\omega(n-k)}\right),
$$
where $T=(T_1,\dots,T_k)\in\Aut(\mathbb B_k)$, $\zeta_j\in\mathbb T$, $j>k$, and $\omega\in\Sigma_{n-k}(\tilde p_{k+1},\dots,\tilde p_n)$.

Let
\begin{align*}
\mathbb F^0_{p,q}&:=\mathbb F_{p,q}\cap\left(\mathbb C^n\times\{0\}^{\tau(1)-1}\times\mathbb C\times\{0\}^{m-\tau(1)}\right),\\
\mathbb F^0_{\tilde p,q_{\tau}/t}&:=\mathbb F_{\tilde p,q_{\tau}/t}\cap\left(\mathbb C^{n+1}\times\{0\}^{m-1}\right).
\end{align*}

Let $\Phi\in\Aut(\mathbb F_{\tilde p,q_{\tau}/t})$ be defined by
$$
\Phi(z,w):=\left(z,U^{-1}(w_1,\dots,w_l),w_{l+1},\dots,w_m\right)
$$
and let
$$
\hat\xi_1:=\begin{cases}\xi_1,\quad&\textnormal{if }l=0\\1,\quad&\textnormal{if }l>0\end{cases},\qquad\hat q_1:=\begin{cases}\tilde q_1,\quad&\textnormal{if }l=0\\1,\quad&\textnormal{if }l>0\end{cases}.
$$
Then $\Phi\circ(G,\psi\circ \Psi_t\circ\tau):\mathbb F^0_{p,q}\longrightarrow\mathbb F^0_{\tilde p,q_{\tau}/t}$ is proper and holomorphic with
\begin{equation}\label{eq:phi1}
(\Phi\circ(G,\psi\circ\Psi_t\circ\tau))(z,w)=\left(G(z,w),\hat\xi_1w_{\tau(1)}^{q_{\tau(1)}/\hat q_1},0,\dots,0\right),\quad (z,w)\in\mathbb F^0_{p,q}.
\end{equation}

It follows from Theorem~\ref{thm:n2m1}~(\ref{item:n2m1form}) that
\begin{equation}\label{eq:phi2}
(\Phi\circ(G,\psi\circ\Psi_t\circ\tau))(z,w)=\left(\hat G(z,w),\eta w_{\tau(1)}^r,0,\dots,0\right),\quad (z,w)\in\mathbb F^0_{p,q},
\end{equation}
where $\hat G=(\hat G_1,\dots,\hat G_n)$,
$$
\hat G_j(z,w):=w_{\tau(1)}^{r\hat q_1/\tilde p_j}\hat f_j\left(z_1w_{\tau(1)}^{-q_{\tau(1)}/p_1},\dots,z_nw_{\tau(1)}^{-q_{\tau(1)}/p_n}\right),\quad j=1,\dots,n,
$$
$\eta\in\mathbb T$, $r\in\mathbb N$, and $\hat f:=(\hat f_1,\dots,\hat f_n):\mathbb E_p\longrightarrow\mathbb E_{\tilde p}$ is proper and holomorphic, i.e.
\begin{equation}\label{eq:hatf}
\hat f=\Psi_{p_{\hat\sigma}/(\tilde p\hat s)}\circ\hat\varphi\circ\Psi_{\hat s}\circ\hat\sigma
\end{equation}
for some $\hat\sigma\in\Sigma_n$ with $p_{\hat\sigma}/\tilde p\in\mathbb N^n$, $\hat s\in\mathbb N^n$ with $p_{\hat\sigma}/(\tilde p\hat s)\in\mathbb N^n$, and $\hat\varphi\in\Aut(\mathbb E_{p_{\hat\sigma}/\hat s})$. Again, it follows from the proof of Theorem~\ref{thm:n2m1}~(\ref{item:n2m1form}) that
$$
\frac{p_{\hat\sigma(j)}}{\hat s_j}=\begin{cases}1,\quad&\textnormal{if }j=1,\dots,k\\\tilde p_j,\quad&\textnormal{if }j=k+1,\dots,n\end{cases},
$$
whence
$$
\hat\varphi(z)=(\hat T(z_1,\dots,z_k),\hat\zeta_{k+1}z_{k+\hat\omega(1)},\dots,\hat\zeta_nz_{k+\hat\omega(n-k)}),
$$
where $\hat T=(\hat T_1,\dots,\hat T_k)\in\Aut(\mathbb B_k)$, $\hat\zeta_j\in\mathbb T$, $j>k$, and $\hat \omega\in\Sigma_{n-k}(\tilde p_{k+1},\dots,\tilde p_n)$.

From (\ref{eq:phi2}) we infer that
\begin{multline}\label{eq:hatG2}
(\Phi\circ(G,\psi\circ\Psi_t\circ\tau))(z,w)=\\
\left(\hat G(z,w)+\alpha(z,w),w_{\tau(1)}^{q_{\tau(1)}},\dots,w_{\tau(l)}^{q_{\tau(l)}},\xi_{l+1}w_{\tau(l+1)}^{q_{\tau(l+1)}/\tilde q_{l+1}},\dots,\xi_mw_{\tau(m)}^{q_{\tau(m)}/\tilde q_m}\right),
\end{multline}
for $(z,w)\in\mathbb F_{p,q}$ with $w_{\tau(1)}\neq0$, where $\alpha$ is holomorphic on $\mathbb F_{p,q}$ with $\alpha|_{\mathbb F^0_{p,q}}=0$. Comparing (\ref{eq:phi1}) and (\ref{eq:phi2}) we conclude that
$$
\eta=\hat\xi_1,\quad r=q_{\tau(1)}/\hat q_1.
$$
Since the mapping on the left side of (\ref{eq:hatG2}) is holomorphic on $\mathbb F_{p,q}$, the functions
$$
\hat G_j(z,w)=\begin{cases}w_{\tau(1)}^{q_{\tau(1)}/\tilde p_j}\hat T_j^{1/\tilde p_j}\left(z_{\hat\sigma(1)}^{p_{\hat\sigma(1)}}w_{\tau(1)}^{-q_{\tau(1)}},\dots,z_{\hat\sigma(k)}^{p_{\hat\sigma(k)}}w_{\tau(1)}^{-q_{\tau(1)}}\right),\hfill&\textnormal{if }j\leq k\\\hat\zeta_jz_{\hat\sigma(j)}^{p_{\hat\sigma(j)}/\tilde p_j},\hfill&\textnormal{if }j>k\end{cases},
$$
are holomorphic on $\mathbb F_{p,q}$, too. Since $m\geq2$, it may happen $w_{\tau(1)}=0$. Consequently, $\hat T\in\mathbb U(k)$ and
$$
\hat G_j(z,w)=\begin{cases}\hat T_j^{1/\tilde p_j}\left(z_{\hat\sigma(1)}^{p_{\hat\sigma(1)}},\dots,z_{\hat\sigma(k)}^{p_{\hat\sigma(k)}}\right),\hfill&\textnormal{if }j\leq k\\ \hat\zeta_jz_{\hat\sigma(j)}^{p_{\hat\sigma(j)}/\tilde p_j},\hfill&\textnormal{if }j>k\end{cases}.
$$

Recall that $\hat G+\alpha=G$ and fix $w\in\{0\}^{\tau(1)-1}\times\mathbb C\times\{0\}^{m-\tau(1)}$ with $0<\rho_w<1$. Then $\rho_w=|w_{\tau(1)}|^{2q_{\tau(1)}}$ and it follows from (\ref{eq:fg}) and (\ref{eq:f}) that
\begin{equation*}
g_j(z)=\begin{cases}|w_{\tau(1)}|^{q_{\tau(1)}/\tilde p_j}T_j^{1/\tilde p_j}\left(z_{\sigma(1)}^{p_{\sigma(1)}}|w_{\tau(1)}|^{-q_{\tau(1)}},\dots, z_{\sigma(k)}^{p_{\sigma(k)}}|w_{\tau(1)}|^{-q_{\tau(1)}}\right),\hfill&\textnormal{if }j\leq k\\ \zeta_jz_{\sigma(j)}^{p_{\sigma(j)}/\tilde p_j},\hfill&\textnormal{if }j>k\end{cases},
\end{equation*}
From the equality $\hat G(\cdot,w)=g$ on $\mathbb E_{p,q}(w)$ one has $\zeta_j=\hat\zeta_j$, $j>k$, and, losing no generality, we conclude that $\sigma=\hat\sigma$, $s=\hat s$, and $T=\hat T$. Consequently, $g$ does not depend on $w$ and $g(0)=0$.
\end{proof}

\bibliographystyle{amsplain}
\bibliography{bib_pz}

\providecommand{\bysame}{\leavevmode\hbox to3em{\hrulefill}\thinspace}
\providecommand{\MR}{\relax\ifhmode\unskip\space\fi MR }
\providecommand{\MRhref}[2]{%
  \href{http://www.ams.org/mathscinet-getitem?mr=#1}{#2}
}
\providecommand{\href}[2]{#2}
\begin{thebibliography}{10}

\bibitem{andreotti1964249}
A.~Andreotti and E.~Vesentini, \emph{On deformations of discontinuous groups},
  Acta Math. \textbf{112} (1964), 249--298.

\bibitem{chen200474}
Z.~H. Chen, \emph{Proper holomorphic mappings between some generalized hartogs
  triangles}, Geometric function theory in several complex variables.
  Proceedings of a satellite conference to the international congress of
  mathematicians, ICM-2002, Beijing, China, August 30--September 2, 2002 (C.~H.
  Fitzgerald and S.~Gong, eds.), River Edge, NJ: World Scientific, 2004,
  pp.~74--81.

\bibitem{chen2003215}
Z.~H. Chen and Y.~Liu, \emph{{A note on the classification of the proper
  mappings between some generalized Hartogs triangles}}, Chin. J. Contemp.
  Math. \textbf{24} (2003), no.~3, 215--220.

\bibitem{chen2008557}
\bysame, \emph{The classification of proper holomorphic mappings between
  special {Hartogs} triangles of different dimensions}, Chin. Ann. Math. Ser. B
  \textbf{29} (2008), no.~5, 557--566.

\bibitem{chen2001177}
Z.~H. Chen and D.~K. Xu, \emph{Proper holomorphic mappings between some
  nonsmooth domains}, Chin. Ann. Math. Ser. B \textbf{22} (2001), no.~2,
  177--182.

\bibitem{chen2002357}
\bysame, \emph{Rigidity of proper self-mapping on some kinds of generalized
  {Hartogs} triangle}, Acta Math. Sin. (Engl. Ser.) \textbf{18} (2002), no.~2,
  357--362.

\bibitem{dini1991219}
G.~Dini and A.~Selvaggi Primicerio, \emph{Proper holomorphic mappings between
  generalized pseudoellipsoids}, Ann. Mat. Pura Appl. (4) \textbf{158} (1991),
  219--229.

\bibitem{isaev200633}
A.~V. Isaev and N.~G. Kruzhilin, \emph{{Proper holomorphic maps between
  Reinhardt domains in $\mathbb C^2$}}, Michigan Math. J. \textbf{54} (2006),
  no.~1, 33--64.

\bibitem{jarnicki2008}
M.~Jarnicki and P.~Pflug, \emph{First steps in several complex variables:
  {Reinhardt} domains}, EMS Textbooks in Mathematics, European Mathematical
  Society Publishing House, 2008.

\bibitem{kosinski2009711}
{\L.~Kosi\'nski}, \emph{{Proper holomorphic mappings between Reinhardt domains
  in $\mathbb C^2$}}, Michigan Math. J. \textbf{58} (2009), no.~3, 711--721.

\bibitem{landucci1984807}
{M.~Landucci}, \emph{On the proper holomorphic equivalence for a class of
  pseudoconvex domains}, Trans. Amer. Math. Soc. \textbf{282} (1984), 807--811.

\bibitem{landucci1989193}
\bysame, \emph{Proper holomorphic mappings between some nonsmooth domains},
  Ann. Mat. Pura Appl. (4) \textbf{155} (1989), 193--203.

\bibitem{stein1972}
E.~M. Stein, \emph{Boundary behavior of holomorphic functions of several
  complex variables}, Mathematical Notes, Princeton University Press, New
  Jersey, 1972.

\bibitem{tu20021035}
Z.~H. Tu, \emph{{Rigidity of proper holomorphic mappings between
  equidimensional bounded symmetric domains}}, Proc. Amer. Math. Soc.
  \textbf{130} (2002), no.~4, 1035--1042.

\bibitem{tu200213}
\bysame, \emph{{Rigidity of proper holomorphic mappings between
  nonequidimensional bounded symmetric domains}}, Math. Z. \textbf{240} (2002),
  no.~1, 13--35.

\bibitem{tu2004310}
\bysame, \emph{Rigidity of proper holomorphic mappings between bounded
  symmetric domains}, Geometric function theory in several complex variables.
  Proceedings of a satellite conference to the international congress of
  mathematicians, ICM-2002, Beijing, China, August 30--September 2, 2002 (C.~H.
  Fitzgerald and S.~Gong, eds.), River Edge, NJ: World Scientific, 2004,
  pp.~310--316.

\bibitem{tu2014703}
Z.~H. Tu and L.~Wang, \emph{{Rigidity of proper holomorphic mappings between
  certain unbounded non-hyperbolic domains}}, J. Math. Anal. Appl. \textbf{419}
  (2014), no.~2, 703--714.

\bibitem{tu20151}
\bysame, \emph{{Rigidity of proper holomorphic mappings between equidimensional
  Hua domains}}, Math. Ann. \textbf{363} (2015), no.~1--2, 1--34.

\end{thebibliography}

\end{document}